\documentclass[a4paper,oneside,11pt]{article}
\usepackage[utf8]{inputenc}
\usepackage[british]{babel} 
\usepackage[textwidth=15cm,top=3cm,bottom=3cm,nohead]{geometry}
\usepackage{amsfonts,amssymb,amsthm,amsmath,amscd, dsfont, mathtools}
\usepackage{enumerate}
\usepackage{tikz}
    \usetikzlibrary{arrows.meta, positioning}

\usepackage[
backend=biber,
style=alphabetic,
sorting=nyt, eprint=false, url=false, isbn=false
]{biblatex}
\usepackage{hyperref}
\hypersetup{
    colorlinks=true,
    linkcolor=blue,
    filecolor=magenta,      
    urlcolor=cyan,
    citecolor = magenta
    }

\newcommand{\bE}{\mathbb{E}}

\newcommand{\bN}{\mathbb{N}}

\newcommand{\bP}{\mathbb{P}}

\newcommand{\cF}{\mathcal{F}}

\newcommand{\cO}{\mathcal{O}}

\newcommand{\fG}{\mathfrak{G}}

\newcommand{\fX}{\mathfrak{X}}

\newcommand{\abs}[1]{\ensuremath{\left\vert#1\right\vert}}

\DeclareMathOperator{\Var}{Var}

\usepackage{graphicx}

\usepackage[capitalise,noabbrev]{cleveref} 
\usepackage{overpic} 
\usepackage{url} 
\usepackage{comment} 

\numberwithin{equation}{section}

\newtheorem{theorem}{Theorem}[section]
\newtheorem{lemma}[theorem]{Lemma}

\newtheorem{proposition}[theorem]{Proposition}

\theoremstyle{definition}
\newtheorem{remark}[theorem]{Remark}

\title{A phase transition in Barak-Erd\H{o}s random graphs}

\author{
    \begin{tabular}{cc}
    Gilles \textsc{Blanchard}\thanks{Universit\'e Paris-Saclay.\hfill  \href{mailto:gilles.blanchard@universite-paris-saclay.fr}{\texttt{gilles.blanchard@universite-paris-saclay.fr}}}
    \qquad\&\qquad &
    Nicolas \textsc{Curien}\thanks{Universit\'e Paris-Saclay.\hfill  \href{mailto:nicolas.curien@gmail.com}{\texttt{nicolas.curien@gmail.com}}} \\
    \&\quad Klara \textsc{Krause}\thanks{Universit\'e Paris-Saclay.\hfill  \href{mailto:klara.krause@universite-paris-saclay.fr}{\texttt{klara.krause@universite-paris-saclay.fr}}}
    \qquad\&\qquad & Alexander \textsc{Reisach}\thanks{Universit\'e Paris-Saclay.\hfill  \href{mailto:alexander.reisach@universite-paris-saclay.fr}{\texttt{alexander.reisach@universite-paris-saclay.fr}}}
    \end{tabular}
}
\date{}

\begin{document}
\maketitle


\begin{abstract}
We study monotone paths in Erd\H{o}s-Rényi random graphs on numbered vertices. This model appeared in \cite{Benjamini2022} where Benjamini \& Tzalik established a phase transition at $p = \frac{\log n}{n}$. We refine the critical value to $p = \frac{\log n - \log \log n }{n}$ and identify the critical window of order $\Theta(1/n)$. 
\end{abstract}

\section{Introduction}

Recall the definition of a Barak-Erd\H{o}s random graph $G_p$ on the positive integers $\bN$: A directed edge $(i,j)$ between any two vertices $i < j \in \bN$ exists independently  with probability $p \in (0,1)$. We denote this law by $\mathbb{P}_p$. We say that there is a monotone path from $1$ to $n$, write $1 \nearrow n$, if there exists a path $1 = i_1 \to i_2 \to \cdots \to i_k = n$ in $G_p$ whith $i_1 < i_2 < \cdots < i_k$. We establish a phase transition for the increasing event 
\[
\left\{1 \nearrow n \textnormal{ in } G_p \right\}.
\]

\begin{theorem}[Critical window] \label{Thm.critical.window}
   The critical window of the event $\left\{ 1 \nearrow n \textnormal{ in } G_{p_n} \right\}$ is of order $\Theta (1/n)$ around $
    \frac{\log n - \log \log n}{n}$. More precisely, if $x \in \mathbb{R}$
    \[
    \mbox{ for } p_{n,x} = \frac{\log n - \log \log n + x}{n} \quad  \mbox{ we have } \quad         \bP _{p_{n,x}}\left( 1 \nearrow n \right) 
        \xrightarrow[n \to \infty]{}  1 - e^{e^{-x} - x} \int_{e^{-x}}^\infty \frac{e^{-t}}{t} \mathrm{d} t.
    \]
\end{theorem}

\begin{remark}
    An alternative expression is 
    \[
    \bP _{p_{n,x}}\left( 1 \nearrow n \right) \xrightarrow[n \to \infty]{} \frac{1}{2} + \frac{1}{2} \bE \left( \tanh\frac{x - \fG}{2} \right)
    \]
    for a standard Gumbel variable $\fG$ with density $p_\fG (z) = \exp(-(z+ e^{-z}))$.
\end{remark}

The location and width of the phase transition might be guessed using the expected number of monotone paths in $G_p$. Indeed, a simple counting exercise shows that
\[
    \bE_{p} \left( \# \textnormal{ paths } 1 \nearrow n \right) = \sum_{k=0}^{n-2} \binom{n-2}{k} p^{k+1} = p(1+p)^{n-2}
\]
and for $p  \equiv p_{n,x}$ the above display converges to $e^x$ as $n \to \infty$.


\begin{figure}[h]
    \centering
    
    \begin{tikzpicture}[
  node distance=0.6cm,
  every node/.style={circle, draw, minimum size=6mm, inner sep=0pt, font=\scriptsize},
  >=Stealth
]

\tikzset{gray vertex/.style={circle, fill=gray!70, inner sep=2pt}}

\node[gray vertex] (1) {1};
\node[below=0.4mm of 1, circle=none, draw=none, font=\small] {$N_1 = 1$};
\node[right=of 1] (2) {2};
\node[right=of 2] (3) {3};
\node[gray vertex, right=of 3] (4) {4};
\node[below=0.4mm of 4, circle=none, draw=none, font=\small] {$N_4 = 2$};
\node[right=of 4] (5) {5};
\node[right=of 5] (6) {6};
\node[right=of 6] (7) {7};
\node[gray vertex, right=of 7] (8) {8};
\node[below=0.4mm of 8, circle=none, draw=none, font=\small] {$N_8 = 3$};
\node[right=of 8] (9) {9};
\node[right= 3mm of 9, circle=none, draw=none] {$\cdots \cdots$};

\draw[->, blue, bend left=30] (1) to (4);  
\draw[->, blue, bend left=25] (2) to (7);  
\draw[->, blue, bend right=30] (3) to (4); 
\draw[->, blue, bend left=25] (5) to (6);  
\draw[->, blue, bend right=25] (4) to (8); 
\draw[->, blue, bend left=25] (6) to (9);  

\end{tikzpicture}

    \caption{A segment from a sampled graph $G_p$ satisfying $1 \nearrow 4$ and $1 \nearrow 8$.}
    \label{fig:}
\end{figure}
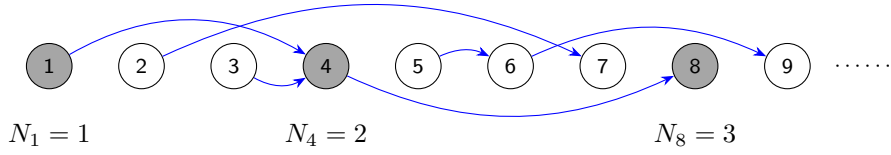


Erd\H{o}s-Rényi graphs on numbered vertices $\{1, \ldots, n\}$ appeared as an ingredient to recover point configurations in \cite{Benjamini2022}. They prove a sharp threshold at $p = \frac{\log n }{n}$ for the event $1 \nearrow n$. Directed acyclic Erdös-Rényi graphs were introduced in \cite{barak.erdos1984} and studied as Barak-Erd\H{o}s graphs since. The articles \cite{foss2003extended}, \cite{mallein2021} and \cite{FoKoMaRa2024} established maximal path lengths in Barak-Erd\H{o}s graphs, both for fixed $p \in (0,1)$ and $p \to 0$.\\
Numbered vertices can be interpreted as time stamps: each vertex becomes available from its assigned time onward. Related models include random temporal graphs, where monotonicity of paths is considered along edges instead of vertices, see for instance \cite{AngelFerber...2018}, \cite{CasteigtsRaskin2024} and \cite{BroutinKamcevLugosi2024}. These works focus on the length of monotone paths and identify several transition phenomena around $p = c \frac{\log n}{n}$ for both longest and shortest such paths. \\
A vertex order appears naturally if the nodes correspond to events unfolding in time, as for example in graphical models of causality \cite{pearl2009causality}. Erd\H{o}s-Rényi graphs directed according to a random vertex order are of particular interest to research in causal discovery \cite{spirtes2001causation}, where they are widely used in the benchmarking of algorithms that learn causal graphs from data.

\paragraph{Acknowledgments}
The first and fourth authors are supported by the ANR grant BISCOTTE - ANR-19-CHIA-0021. The second and third authors are supported by the SuPerGRandMa, the ERG Consolidator no 101087572.

\section{Proof}

The central ingredient to prove Theorem \ref{Thm.critical.window} is the introduction of an exploration process. It indicates the position of having passed a certain number of vertices that are reachable from $1$ by a monotone path. We show that, after suitable centering and scaling, this exploration process converges to a Gumbel variable with a deterministic drift in Section \ref{Sec.exploration.process}. This allows to compute explicitly the probability that a monotone path $1 \nearrow n$ exists.


Define 
\[
    N_k \coloneq \# \left\{1 \leq i \leq k: 1 \nearrow i\right\}
\]
the number of vertices up to $k$ that are reachable from $1$ by a monotone path and set $\cF_k = \sigma(N_1, \ldots, N_k)$. A new vertex $k$ is reachable from $1$ precisely when there is an edge to at least one of the $N_{k-1}$ previously reachable vertices. Hence 
\[
    \bP_p \left(N_{k} = N_{k-1} +1 \vert \cF_{k-1}\right) 
    = 1 - \bP_p\left( N_{k} = N_{k-1} \vert \cF_{k-1} \right) 
    = 1 - (1-p)^{N_{k-1}}
\] 
and 
\begin{equation}\label{Eq.Proba.1.to.k}
    \bP_p(1 \nearrow k) = \bE \left( 1-(1-p)^{N_{k-1}} \right).
\end{equation}
The probability that successive gaps between reachable vertices have sizes $k_1, \ldots,k_l \geq 1$ is 
\[
\bP_p\!\left(
\overset{1}{\bullet}\,
\underbrace{\cdots}_{k_1}\,
\bullet\,
\underbrace{\cdots}_{k_2}\,
\bullet\,
\cdots \cdots \,
\bullet\,
\underbrace{\cdots}_{k_l}\,
\bullet
\right)
 = \prod_{i=1}^l ((1-p)^i)^{k_i} \left( 1- (1-p)^i \right)
\]
where each large bullet marks a vertex reachable from $1$ and each group of small dots represents the vertices in the corresponding gap: The $i$-th gap is a geometric variable with parameter $1-(1-p)^i$, and the gaps are pairwise independent.


\subsection{Exploration process} \label{Sec.exploration.process}

Fix $p \in (0,1)$ and consider the monotone exploration process $(P_p(a))_{a> 0}$ where
\[
P_p(a) \coloneq \inf \left\{ k \geq 1: N_k = \lfloor a/p \rfloor \right\}
\] 
denotes the position of the $\lfloor a/p \rfloor$-th vertex reachable from $1$ by a monotone path.  
Since $P_p(a)$ is obtained by summing the successive gap lengths, we have  
\[
P_p(a) = \sum_{i=1}^{\lfloor a/p \rfloor} X_i, \quad X_i \sim \mathrm{Geom}\left(1-(1-p)^i\right) \text{ independent}. 
\]
The exploration process $(P_p(a))_{a> 0}$ centered and rescaled, converges, as $p \to 0$, to a standard Gumbel variable $\fG$ with a deterministic drift depending on $a$.


\begin{proposition}[Convergence of the monotone exploration process] \label{Prop.convergence.to.Gumbel}
   We have \[ \left( p P_p(a) - \log \frac{1}{p} \right)_{a > 0} \xrightarrow[p \to 0]{(d)} \left( \fG + \log(e^a -1) \right)_{a > 0} \] 
for the uniform convergence over every compact subset of $(0, \infty)$.
\end{proposition}


\begin{proof}[Proof of Proposition \ref{Prop.convergence.to.Gumbel}]
    Define $q_i \coloneq 1-(1-p)^i$. Assume without loss of generality that $p < 1/2$. For any $p< 1/2$ and $i \leq \lfloor a/p \rfloor$ it holds $q_i\leq 1- (1-p)^{\lfloor a/p \rfloor}\leq B(a)<1$ for some constant $B(a)$ independent of $p$. Each $P_p(a)$ is the sum of independent geometric variables with expectation $1/q_i$. \\
   We first approximate the rescaled geometric variables $q_iX_i$ by i.i.d. exponential variables via a coupling: Let $(U_i)_i$ be i.i.d. uniform variables on $[0,1]$. Set 
    \[
    Y_i = - \log U_i \, \textnormal{ and } \, X_i = \inf\left\{k \geq 1: U_i \geq (1-q_i)^k\right\}
    \] 
    so that $Y_i$ are exponential variables with intensity $1$ and $X_i$ are geometric variables with mean $1/q_i$. Then $ \log \left(\frac{1}{1-q_i}\right) (X_i - 1)  \leq Y_i $ which allows the bound
    \begin{equation}\label{couplingbound}
        \abs{q_iX_i - Y_i} \leq \abs{ \frac{q_i}{\log(1/(1-q_i))} -1 } Y_i + q_i \leq C(a) q_i (Y_i +1)
    \end{equation}
    for some constant $C(a)>0$, using that $q_i < B(a) <1$.\\

    In a second step, we express the exploration process using the exponential variables $Y_i$ and show that the occurring error has converging expectation and vanishing variance, that is
    \begin{align*}
        p P_p(a) - \log \frac{1}{p} &\overset{(d)}{=} \sum_{i=1}^{\lfloor a/p \rfloor} \frac{1}{i} Y_i - \log\frac{1}{p} + R_p(a) 
    \end{align*}
    with rest term \[R_p(a) = \sum_{i=1}^{\lfloor a/p \rfloor} p \left( \frac{1}{q_i} - \frac{1}{ip} \right) Y_i + \sum_{i=1}^{\lfloor a/p \rfloor} \frac{p}{q_i} (q_iX_i - Y_i).\]  

    \begin{lemma}\label{Lemma.rest.term}
        As $p\to 0$, the rest term satisfies 
        \[
        \bE(R_p(a)) = \log (e^a - 1) - \log a + o(1) \, \text{ and } \, \Var(R_p(a)) = o(1).
        \] 
    \end{lemma}
    In a last step, we use the classical fact that the sum of the rescaled exponential variables, properly centered, converges in law to a standard Gumbel variable,
    \begin{equation} \label{Eq.Exponentials.to.Gumbel}
        \sum_{i=1}^{\lfloor a/p \rfloor} \frac{1}{i} Y_i - \log \frac{1}{p} \xrightarrow[p \to 0]{(d)} \fG + \log a,
    \end{equation}
    see e.g.~\cite[Lemma 11.2]{Curien2025}. 
    Therefore
    \begin{align*}
        p P_p(a) - \log \frac{1}{p} &\xrightarrow[p \to 0]{(d)} \fG + \log (e^a - 1)
    \end{align*}
    for any fixed $a>0$. In fact, this is the same Gumbel variable for any two $a_1, a_2 > 0$.\\
    We now conclude the convergence as a process in $a>0$ via a probabilistic version of Dini's lemma, see \cite[Theorem 11.6]{Curien2025}:
    The process \[\left( pP_p(a) - \log \frac{1}{p} - \fG\right)_{a > 0}\] is increasing in $a$ and for every fixed $a> 0$ it converges in law to $\log (e^a -1)$ which is an increasing continuous function. Therefore Dini's lemma implies the convergence as a process \[\left( pP_p(a) - \log \frac{1}{p} - \fG\right)_{a > 0} \xrightarrow[p \to 0]{(d)} \left(\log(e^a-1 ) \right)_{a > 0}\] for the topology of uniform convergence over every compact subset of $(0,\infty).$
    \end{proof}


It remains to prove Lemma \ref{Lemma.rest.term}.

\begin{proof}[Proof of Lemma \ref{Lemma.rest.term}]

For the expectation notice that $\bE(q_iX_i) = 1 = \bE(Y_i)$ so that
\begin{align*}
    \bE(R_p(a)) &= \bE \left( \sum_{i=1}^{\lfloor a/p \rfloor} p \left( \frac{1}{q_i} - \frac{1}{ip} \right) Y_i \right) 
    = p \sum_{i=1}^{\lfloor a/p \rfloor} \frac{1}{q_i} - \frac{1}{ip} \\
    &= p \sum_{i=1}^{\lfloor a/p \rfloor} \left( \frac{1}{1-e^{-ip}} - \frac{1}{ip} \right) + p \sum_{i=1}^{\lfloor a/p \rfloor} \left(\frac{1}{q_i} - \frac{1}{1-e^{-ip}} \right) \\
    & \leq\int_0^a\left( \frac{1}{1-e^{-s}} -\frac{1}{s} \right) \mathrm{d} s + p  \sum_{i=1}^{\lfloor a/p \rfloor} \frac{C(a)}{i} \\
    &= \log (e^a-1) - \log a + o(1)
\end{align*}
for some constant $C(a)>0$ depending on $a$.
For the variance write 
\[
Z_i = \left( \frac{1}{q_i} - \frac{1}{ip} \right) Y_i + \frac{1}{q_i} (q_iX_i - Y_i) \leq c_1(a) Y_i + c_2(a)
\]
for some constants $c_1(a) > 0$ and $c_2(a) > 0$, using the bound $\left(\frac{1}{q_i} - \frac{1}{ip}\right) \leq \frac{a}{2(1-e^{-a})} $ for $i \leq \lfloor a/p \rfloor$ and (\ref{couplingbound}). Then, as $Y_i \sim \operatorname{Exp}(1)$ for all $i$, this yields 
\[
 \Var (Z_i) \leq \bE\left(Z_i ^2\right) = \cO(1).
\]
Independence of the $Z_i $ thus yields 
\[
 \Var \left( R_p(a) \right) = p^2 \sum_{i=1}^{\lfloor a/p \rfloor} \Var (Z_i ) = o(1).
\]
\end{proof}


\subsection{Critical window} \label{Sec.proof.prob}


We now choose $p$ dependent on $n$ and let $n \to \infty$. More precisely, let $b > 0$ and let $p = p_n$ satisfy 
\begin{equation} \label{Eq.choice.pn}
    n = \frac{\log (b/p_n)}{p_n}.
\end{equation}


\begin{proof}[Proof of Theorem \ref{Thm.critical.window}]
    For each $n \in \bN$ define the random variable $\fX_n$ as the number of vertices smaller than $n$ that are reachable from $1$ by a monotone path, scaled by $p_n$, that is 
    \[
    P_{p_n} (\fX_n) = n \quad \text{with} \quad \lfloor \fX_n/p_n \rfloor = N_n.
    \] 
    Then, by (\ref{Eq.Proba.1.to.k})
    \begin{equation} \label{Eq.1.to.n.expressed.by.X}
         \bP_{p_n}\left( 1 \nearrow n \right) = \bE\left(1-(1-p_n)^{\lfloor \fX_n/p_n \rfloor -1}\right) = \bE \left(1 - e^{-\fX_n} (1 + o(1)) \right).
    \end{equation}
    On the other hand, by the convergence in law from Proposition \ref{Prop.convergence.to.Gumbel} and the Skorokhod representation theorem there is a probability space with a variable $\tilde{\fG}$ and a process $\tilde{P}_{p_n}$ inducing variables $\tilde{\fX}_n$ with the same law as $\fG$, $P_{p_n}$ and $\fX_n$ respectively such that almost-sure-convergence
    \[
        np_n - \log \frac{1}{p_n} - \log \left( e^{\tilde{\fX}_n} -1 \right) \xrightarrow[n \to \infty]{} \tilde{\fG} \quad \text{a.s.}
    \]
    holds. Therefore
    \begin{equation} \label{Eq.X.as.function.of.Gumbel}
        1 - e^{-\tilde{\fX}_n}  \xrightarrow[n \to \infty]{} 1 - \frac{1}{1+ be^{-\tilde{\fG}}} \quad \text{a.s.}
    \end{equation}
    and this convergence holds in law for $\fX$ and $\fG$. Combine (\ref{Eq.1.to.n.expressed.by.X}) and (\ref{Eq.X.as.function.of.Gumbel}) to compute
    \begin{equation}\label{Eq.nontrivial.probability.in.terms.of.b}
        \bP_{p_n}\left( 1 \nearrow n \right) 
        = 1- \bE \left( \frac{1}{1 + b e^{-\fG}} \right) + o(1)
        = 1 - \frac{1}{b}e^{1/b} \int_{1/b}^\infty \frac{e^{-t}}{t} \mathrm{d} t + o(1). 
    \end{equation}
    The function $f(b) = \frac{1}{b}e^{1/b} \int_{1/b}^\infty \frac{1}{t} e^{-t} \mathrm{d}  t$ is continuous on $(0, \infty)$ with asymptotics
    \[
    f(b)  \xrightarrow{b \to \infty} 0 \quad \text{and} \quad f(b) \xrightarrow{b \to 0} 1.
    \]
    
    So for any $b> 0$ the choice $p_n$ in (\ref{Eq.choice.pn}) yields a non-trivial probability for $\{1 \nearrow n \}$. It remains to determine the matching value $b = b_x$ when choosing $p_{n,x}$ as in the theorem. From (\ref{Eq.choice.pn}) we have 
    \[
    b_x = p_{n,x}e^{np_{n,x}} = \left(\frac{\log n - \log \log n}{n}+ \frac{x}{n}\right) \frac{n}{\log n} e^x = \left( 1 + \frac{x - \log \log n}{\log n} \right) e^x \sim e^x.
    \] 
    Plug this into equation (\ref{Eq.nontrivial.probability.in.terms.of.b}).
\end{proof}

%

\begin{remark}[Number of paths]

    Denote by $M_a^+$ the number of paths $1 \nearrow P_p(a)$, conditioned to be positive. We conjecture that, as $p \to 0$, $M_a^+$ converges to a geometric random variable with parameter $e^{-a}$. Heuristically, it satisfies the recursive relation
    \[
        M_a^+ \overset{(d)}{=} \sum_{i=1}^W M_{aU_i}^+ 
    \]
    where $W$ is a $\mathrm{Poisson}(a)$ random variable conditioned to be positive, and $(U_i)_{i \geq 1}$ are i.i.d $\mathrm{Unif}[0,1]$ variables.
\end{remark}

\printbibliography
\end{document}